\documentclass[12pt]{article}
\usepackage{amsmath,amsxtra,amssymb,latexsym, amscd,amsthm}

\advance\voffset-1truecm\relax
\advance\hoffset-1truecm\relax

\def\R{{\Bbb R}}

\def\R{{\Bbb R}}
\def\C{{\Bbb C}}
\def\P{{\Bbb P}}

\numberwithin{equation}{section}

\newcommand {\Cal}{\mathcal}

\newtheorem{lemma}{Lemma}[section]
\newtheorem{theorem}[lemma]{Theorem}

\def\leq{\leqslant}

\begin{document}
\title{{Higher dimensional generalizations of some  theorems on normality of meromorphic functions}}
\author{Tran Van Tan}
\date{\small }
\maketitle
\begin{abstract} In [Israel J. Math, 2014],  Grahl and Nevo obtained a significant improvement for the well-known normality criterion of Montel. They  proved that for a family of meromorphic functions $\mathcal F$  in a domain $D\subset \C,$ and for a positive constant $\epsilon$,  if for each $f\in \mathcal F$ there exist meromorphic functions $a_f,b_f,c_f$ such that $f$ omits $a_f,b_f,c_f$ in $D$ and $$\min\{\rho(a_f(z),b_f(z)), \rho(b_f(z),c_f(z)), \rho(c_f(z),a_f(z))\}\geq \epsilon,$$ for all $z\in D$, then  $\mathcal F$ is normal in $D$. Here, $\rho$ is the spherical metric in $\widehat\C$. In this paper, we establish the high-dimensional versions for the above result and for the following well-known result of Lappan:  A meromorphic function $ f$  in the unit disc
$\triangle:=\{z\in\C: |z|<1\}$  is normal if there are five distinct values $a_1,\dots,a_5$ such that
$$\sup\{(1-|z|^2)\frac{ |f '(z)|}{1+|f(z)|^2}: z\in f^{-1}\{a_1,\dots,a_5\}\} < \infty.$$

\end{abstract}
\section{Introduction}
Perhaps the most celebrated theorem in the theory of normal families is the following criterion of Montel \cite{M}.
\\

\noindent{\bf Theorem A.} {\it Let $\mathcal F$ be a family of meromorphic functions in a domain $D\subset\C$, and let $a,b,c$ be three distinct points in $\widehat\C.$
Assume that all functions in $\mathcal F$ omit three points $a,b,c$ in $D.$ Then $\mathcal F$ is a normal family in $D.$}\\

In  \cite{C}, Carath\'eodory extended Theorem A to the case where the omitted  points may depend on the  function in the family and satisfy a condition on the spherical distance. In 2014, Grahl and Nevo \cite{GN} generalized the result of Carath\'eodory to the case where all functions in the family omit three functions, and obtained  the following theorem.\\

\noindent{\bf Theorem B.}  {\it Let $\mathcal F$ be a family of meromorphic functions in a domain
$D\subset\C,$ and let $\epsilon$ be a positive constant.  Denote by $\rho$ the spherical metric in $\widehat\C.$ Assume that for each $f\in \mathcal F$ there exist meromorphic functions $a_f,b_f,c_f$ in $D$ such that $f$ omits $a_f,b_f,c_f$ in $D$ and $$\min\{\rho(a_f(z),b_f(z)), \rho(b_f(z),c_f(z)), \rho(c_f(z),a_f(z))\}\geq \epsilon,$$ for all $z\in D$, $f\in\mathcal F.$ Then  $\mathcal F$ is a normal family in $D$.}\\

\noindent In section 3, we shall establish the higher dimensional version of Theorem B.

A well-known result of Lehto and Virtanen \cite{LV} states that
the meromorphic function $f$ in the unit disc $\triangle:=\{z\in\C: |z|<1\}$  is normal if and only
if $\sup_{z\in\triangle}(1-|z|^2) f ^{\#}(z) < \infty$, where $f ^{\#}:=\frac{|f'|}{1+|f|^2}$ is the spherical
derivative of $f.$

In 1972, Pommerenke \cite{P} gave an open question:  if $M > 0$ is given, does there exist a finite subset $E\subset\widehat\C$ such
that if $f $ is a meromorphic in $\triangle$ then the condition that $(1 - |z|^2) f ^{\#}(z) \leq M$  for each
$z \in f ^{-1}(E)$  implies that $f $ is a normal function?
This question was answered by the  following well-known result  of Lappan \cite{L1}.\\

\noindent{\bf Theorem C.} Let $E\subset\widehat{\C}$ be any set consisting  of five distinct  points. If $ f$ is a meromorphic function in  $\triangle$  such that
\begin{align}\label{Lap}\sup\{(1-|z|^2) f ^{\#}(z): z\in f^{-1}(E)\} < \infty
\end{align}
then $f$ is a normal function.\\

In 1986, Hahn \cite{H} generalized Theorem C to the case of high dimension, however, unfortunately his proof  based on a false lemma (Lemma 2, \cite{H}).

In section 4, we shall establish the higher dimensional version for the above five-point theorem of Lappan.

In the case of dimension one, if $f$ is not normal, then by a result of Lohwater and Pommerenke \cite{LP}, there exist   sequences $\{z_k\}\subset\triangle,$ $\{r_k\}\subset \R,$ $r_k>0,$ with $\lim_{k\to\infty}\frac{r_k}{1-|z_k|}=0$ such that $g_k(\xi):=f(z_k+ r_k\xi)$ converges  uniformly on compact subsets of $\C$ to a non-constant meromorphic function $g.$ Then  condition (\ref{Lap}) implies that all zero points of $g-a \;(a\in E)$ have multiplicity at least 2; this is impossible  because  that $g$ is non-constant and $\#E=5.$

In the high dimensional case ($n\geq 2),$ from the view of Nevanlinna theory,  the most difficulty comes from the fact that for any $q\geq n+1$, there are hyperplanes $H_1,\dots,H_q$ in general position in $P^n(\C)$, and a non-constant entire curve $g$ in $P^n(\C)$ such that all zero points of $H_j(g)$ $(j=1,\dots,q)$ have multiplicity not less than 2 (even not less than $n$). Indeed, let $u$ be a non-constant holomorphic function  nowhere vanishing on $\C$. We consider $g=(\binom{0}{n}u^n:\binom{1}{n}u^{n-1}:\cdots:\binom{n}{n}u^0)$ and $q \;(q\geq n+1)$ hyperplanes  $H_j: a_j^0x_0+a_j^1x_1+\cdot +a_j^nx_n =0,$ ($j=1,\dots,q)$, where $a_1,\dots,a_q$ are $q$ distinct complex numbers.  Then they satisfy:

 i) $g$  is linearly non-degenerate;

 ii) For any $1\leq j_0<j_1<\cdots<j_n\leq q,$ the Vandermonde determinant $\det(a_{j_i}^s)_{0\leq i,s\leq n}=\prod_{0\leq t<k\leq n}(a_{j_k}-a_{j_t}) \ne0,$ hence, $H_1,\dots,H_q$ are in general position;

  iii) $H_j(g)=(u+a_j)^n$, hence, all zero points of $H_j(g)$ $ (j=1,\dots,q$) have multiplicity not less than $n.$

The above example also shows that Lemma 2 in \cite{H} is false.
\\

\noindent {\bf Acknowledgements:} This research was supported by Vietnam National Foundation for Science and Technology Development (NAFOSTED) under grant number 101.02-2016.17, and  was done during a stay of the author at  the Vietnam Institute for Advanced Studies in Mathematics. He wishes to express his gratitude to this institute. The author also  would like to thank the referee for  valuable comments and suggestions.
\section{Notations}
Let $\nu$ be a nonnegative divisor on $\C.$ For each positive integer (or $+\infty) $ $ p,$ we define the
counting function of $\nu$ (where multiplicities are truncated by $p)$ by
\begin{align*}N^{[p]}(r, \nu) :=\int_{1}^r\frac{n_{\nu}^{[p]}}{t}dt
\quad (1 < r < \infty)
\end{align*}
where $n_{\nu}^{[p]}(t)=\sum_{|z|\leq t}\min\{\nu(z), p\}.$ For brevity we will omit the character $ [p]$  in the counting
function  if $p = +\infty.$

For a  meromorphic function $\varphi$  on $\C$ $(̣\varphi\not\equiv 0, \varphi\not\equiv\infty$),  we denote by $(\varphi)_0$ the divisor of zeros of $\varphi$. We have the following Jensen's formula for the counting function:
\begin{align*} N(r, (\varphi)_0)-N(r, \left(\frac{1}{\varphi}\right)_0)=\frac{1}{2\pi}\int_{0}^{2\pi}\log\left|\varphi(re^{i\theta})\right|d\theta +O(1).
\end{align*}
We define the proximity function of $\varphi$ by
\begin{align*} m(r, \varphi)=\frac{1}{2\pi}\int_{0}^{2\pi}\log^{+}\left|\varphi(re^{i\theta})\right|d\theta,
\end{align*}
where $\log^+x=\max\{0,\log x\}$ for $x\geq0.$

 If $\varphi$ is nonconstant then $m(r,\frac{\varphi'}{\varphi})=o(T_{\varphi}(r))$ as $r\to\infty, $ outside a set of finite Lebesgue measure  (Nevanlinna's lemma on the logarithmic derivative).

Nevanlinna's first main theorem for $\varphi$ states that $T_\varphi(r)=N_{\frac{1}{\varphi}}(r)+m(r,\varphi) +O(1).$

Let $f$ be a holomorphic mapping of $\C$ into $P^n(\C)$ with a reduced representation $(f_0,\dots,f_n).$ The characteristic  function $T_f(r)$ of $f$ is defined by
\begin{align*}T_f(r):=\frac{1}{2\pi}\int_{0}^{2\pi} \log\Vert f(re^{i\theta})\Vert d\theta-\frac{1}{2\pi}\int_{0}^{2\pi} \log\Vert f(e^{i\theta})\Vert d\theta,\quad r>1,
\end{align*}
where $\Vert f\Vert=\max\limits_{i=0,\dots,n}|f_i|.$

Let   $H=\{(\omega_0:\cdots:\omega_n)\in P^n(\C): a_0\omega_0+\cdots+a_n\omega_n=0\}$  be a hyperplane in $P^n(\C)$ such that $f(\C)\not\subset H.$ Denote by $(H(f))_0$ the divisor of zeros of $a_0f_0+\cdots+a_nf_n,$ and put $N_f^{[p]}(r,H):=N^{[p]}(r, (H(f))_0).$

Let $q,\kappa$ be positive integers, $q\geq \kappa\geq n$ and let  $H_1,\dots,H_q$ be hyperplanes in $P^n(\C).$ These hyperplanes are said to be in $\kappa$-subgeneral  position if $\cap_{i=0}^\kappa H_{j_i}=\varnothing,$ for all $1\leq j_0<\cdots<j_\kappa\leq q.$

\noindent {\bf Nochka's second main theorem.} Let $f$ be a linearly nondegenerate holomorphic mapping of $\C$ into $P^n(\C),$ and let $H_1,\dots, H_q$ be hyperplanes  in $\kappa$-subgeneral position in $P^n(\C)$ ($\kappa\geq n$ and $q\geq 2\kappa-n+1).$ Then
\begin{align*}(q-2\kappa+n-1)T_f(r)\leq \sum_{j=1}^qN_f^{[n]}(r,H_j)+o(T_f(r)),
\end{align*}
for all $r\in(1,+\infty)$ excluding a subset of finite Lebesgue measure.

\section{The high-dimensional version of the Grahl-Nevo theorem}
Let $D$ be a  domain in $\C^m,$ and let $f$ and $H$ be two holomorphic mappings of $D$ into $P^n(\C).$ For each $z\in D$, we take  reduced representations $\widehat f=(f_0,\dots,f_n)$ of $f$ and $\widehat H=(a_0,\dots,a_n)$ of $H$ in a neighbourhood $U$ of $z$ and set $\big<\widehat f, \widehat H\big>:=a_0f_0+\cdots+a_nf_n.$ Denote by $\big<\widehat f,\widehat H\big>_0$  the zero divisor of the holomorphic function $\big<\widehat f, \widehat H\big>. $ The divisor $(H(f))_0:= \big<\widehat f,\widehat H\big>_0
$ is determined independently of a choice of reduced representations, and hence is well defined on the totality of $D.$  Put $f^{-1}(H):=\{z\in D: (H(f))_0(z)>0\}.$

 For $n+1$ points  $a_0,\dots,a_n$ in $P^n(\C),$ we denote by $d_{FS}(a_0,\dots,a_n)$ the  minimum of the Fubini-Study distances   from each point  to the subspace generated by these $n$ other points.

 We shall prove the following normality criterion.
\begin{theorem}\label{Normal} Let $\mathcal F$ be a family of holomorphic mappings of a domain $D\subset \C^m$ into $P^n(\C).$ For each $f\in \mathcal F$, we consider $2n+1$ holomorphic mappings $H_{1f},\dots, H_{(2n+1)f}$  of $D$ into $P^n(\C)$ satisfying the following condition:

\noindent For each compact subset $K$ of $ D,$ there is a positive constant $\delta_K$ such that
 \begin{align*}d_{FS}(H_{j_0f}(z),\dots,H_{j_nf}(z))\geq \delta_K
 \end{align*}
  for all subsets $\{j_0,\dots, j_n\}\subset\{1,\dots 2n+1\}$ and all $z\in K,$ $f\in\mathcal F.$

Assume that $f^{-1}(H_{jf})=\varnothing$ for all $f\in\mathcal F$ and $j\in\{1,\dots,2n+1\}.$
Then $\mathcal F$ is  normal on $D.$
\end{theorem}
\begin{lemma}[Zalcman lemma, \cite{AK}, Lemma 3.1] \label{Zalcman}Let $\mathcal F$ be a family of holomorphic mappings of a domain $D\subset \C^m$ into $P^n(\C).$ If $\mathcal F$ is not normal then there exist  sequences $\{z_k\}\subset D$ with $z_k\to z_0\in D,$  $\{f_k\}\subset\mathcal F$, $\{\rho_k\}\subset\R$ with $\rho_k\to 0^+,$ and
Euclidean unit vectors $\{u_k\}\subset \C^m,$ such that $g_k(\zeta) := f_ k(z_k + \rho_k u_k\zeta),$ where $\zeta\in\C$
satisfies $z_k+ \rho_k u_k\zeta\in D,$ converges uniformly on compact subsets of $\C$ to a nonconstant
holomorphic mapping $g$ of $\C$ into $P^n(\C).$
\end{lemma}
\begin{lemma}[\cite{Ch}, Corollary 14]\label{Cherry} Let $P_0=(\omega_{00}:\cdots:\omega_{0n}),\dots,P_n=(\omega_{n0}:\cdots:\omega_{nn})$ be $n+1$ points in $P^n(\C).$ Then
\begin{align*}d_{FS}^n(P_0,\dots,P_n)\leq \frac{\left|\det(P_0,\dots,P_n)\right|}{\Vert P_0\Vert\cdots\Vert P_n\Vert}\leq d_{FS}(P_0,\dots,P_n),
\end{align*}
where $\Vert P_j\Vert=(|\omega_{j0}|^2+\cdots+|\omega_{jn}|^2)^{\frac{1}{2}}$ and $\det(P_0,\dots,P_n):=\det(\omega_{ji})_{0\leq i,j\leq n}.$
\end{lemma}
\noindent In fact, in (\cite{Ch}, Corollary 14),  the points $P_0,\dots,P_n$ are    projectively independent, however, if they are projectively dependent, then $$\det(P_0,\dots,P_n)=0= d_{FS}(P_0,\dots,P_n).$$

 In the case  $n=m=1$,  the following lemma is due to Grahl and  Nevo \cite{GN}.
\begin{lemma}\label{tan} Let $\{H_{1\alpha}\}_{\alpha\in\mathcal A}, \dots,\{H_{q\alpha}\}_{\alpha\in\mathcal A}$ be $q$ ($q\geq n+1)$ families of holomorphic mappings of $D\subset\C̉̉̉̉̉̉̉^m$ into $P^n(\C).$ Assume that for each compact subset $K$ of $D$, there is a positive constant $\delta_K$ such that
$$d_{FS}(H_{j_0\alpha}(z),\dots,H_{j_n\alpha}(z))\geq \delta_K,$$
for all   $z\in K$, $\alpha\in\mathcal A$, and $1\leq j_0<j_1<\cdots< j_n\leq q$.

\noindent Then $\{H_{1\alpha}\}_{\alpha\in\mathcal A}, \dots,\{H_{q\alpha}\}_{\alpha\in\mathcal A}$ are normal families on $D$.
\end{lemma}
\begin{proof} Suppose that there is an index $j\in\{1,\dots, q\}$ such that $\{H_{j\alpha}\}_{\alpha\in\mathcal A}$ is not normal on $D,$ say $j=1.$
By induction, we prove the following claim:
For each $s\in\{1,\dots,q\},$  there exist  sequences $\{\alpha_k\}_{k=1}^\infty\subset\mathcal A,$   $\{z_{k}\}\subset D$ with $z_{k}\to a\in D,$   $\{\rho_{k}\}\subset\R$ with $\rho_{k}\to 0^+,$ and
Euclidean unit vectors $\{u_{k}\}\subset \C^m,$ such that for all $j\in\{1,\dots,s\},$  $H_{j,k}(\zeta) := H_ {j\alpha_{k}}(z_{k} + \rho_{k} u_{k}\zeta),$ where $\zeta\in\C$
satisfies $z_{k}+ \rho_{k} u_{k}\zeta \in D,$ converges uniformly on compact subsets of $\C$ to a
holomorphic mapping $L_j$ of $\C$ into $P^n(\C),$ where at least one of $L_1,\dots, L_s$ is nonconstant.

The case $s=1$ is just Lemma \ref{Zalcman}.

Assume that the claim is true for some $s\in\{1,\dots,q-1\},$ we prove that it holds for  $s+1.$ By the induction hypothesis, there exist  sequences $\{\alpha_k\}_{k=1}^\infty\subset\mathcal A,$   $\{z'_{k}\}\subset D$ with $z'_{k}\to a\in D,$   $\{\rho'_{k}\}\subset\R$ with $\rho'_{k}\to 0^+,$ and
Euclidean unit vectors $\{u_{k}\}\subset \C^m,$ such that for all $j\in\{1,\dots,s\},$  $H_{j,k}(\zeta) := H_ {j\alpha_{k}}(z'_{k} + \rho'_{k} u_{k}\zeta),$ where $\zeta\in\C$
satisfies $z'_{k}+ \rho'_{k} u_{k}\zeta \in D,$ converges uniformly on compact subsets of $\C$ to a
holomorphic mapping $L_j$ of $\C$ into $P^n(\C),$ where at least one of $L_1,\dots, L_s$ is nonconstant.

We consider the sequence $H_{s+1, k}(\zeta):=H_ {(s+1)\alpha_{k}}(z'_{k} + \rho'_{k} u_{k}\zeta),$ where $\zeta\in\C$
satisfies $z'_{k}+ \rho'_{k} u_{k}\zeta \in D.$

 If $\{H_{s+1,k}\}_{k=1}^\infty$ is normal on $\C,$ then by replacing by an appropriate subsequence, without loss of generality, we assume that  $H_{s+1,k}$ converges uniformly on compact subsets of $\C$ to a
holomorphic mapping $L_{s+1}$ of $\C$ into $P^n(\C).$ Hence, in this case, combining with the induction hypothesis, we get that  the claim is also true for $p+1.$

 If $\{H_{s+1, k}\}_{k=1}^\infty$ is not normal on $\C,$ then by Lemma \ref{Zalcman}, there exist   a subsequence  of $\{H_{s+1,k}\}_{k=1}^\infty$ which without loss of generality we also denote by $\{H_{s+1,k}\}_{k=1}^\infty$ and sequences $\{\xi'_{k}\}\subset \C$ with $\xi'_{k}\to \xi^0\in \C,$   $\{t_{k}\}\subset\R$ with $t_{k}\to 0^+,$
 such that  $h_{s+1,k}(\zeta) :=H_{s+1,k}(\xi'_k+t_k\zeta)= H_ {(s+1)\alpha_{k}}(z'_{k}+\rho'_k\xi'_ku_k + \rho'_{k}t_k u_{k}\zeta)$ converges uniformly on compact subsets of $\C$ to a
nonconstant holomorphic mapping $L_{s+1}$ of $\C$ into $P^n(\C).$ Set $z_k:=z'_{k}+\rho'_k\xi'_ku_k$, $\rho_k:=\rho'_{k}t_k$. Then $z_k\to a,$  $\rho_k\to 0^+$, $h_{s+1,k}(\zeta) = H_ {(s+1)\alpha_{k}}(z_{k} + \rho_{k} u_{k}\zeta)$ converges uniformly on compact subsets of $\C$ to a
nonconstant holomorphic mapping $L_{s+1}$ of $\C$ into $P^n(\C)$, and $h_{j,k}(\zeta) := H_{j,k}(\xi'_k+t_k\zeta)=H_ {j\alpha_{k}}(z_{k} + \rho_{k} u_{k}\zeta)$  converges uniformly on compact subsets of $\C$ to the point $L_j(\xi^0),$ for all $j\in\{1,\dots,s\}$ (note that $H_{j,k}\to L_j$ and $\xi'_k\to\xi^0$). Therefore,  the claim is true for $p+1.$

\noindent By induction, we get the claim.

In our claim, without loss of generality, we assume that $L_1$ is nonconstant.
Take a ball $B(a,r):=\{z: \Vert z-a\Vert \leq r\}\subset D,$ for some $r>0$ (note that $a\in D$).  By the assumption, there is a constant $\delta>0$ such that $d_{FS}(H_{1\alpha}(z),\dots,H_{(n+1)\alpha}(z))\geq \delta$  for all $z\in K,$ $\alpha\in\mathcal A.$ \\
For each $\zeta\in \C$, it is clear that $z_{k} + \rho_{k} u_{k}\zeta\in B(a,r)$ for all  $k$ sufficiently large.
Hence, by our above claim and by Lemma \ref{Cherry}, we have

\begin{align*}& \frac{\left|\det\left(L_1(\zeta),\dots, L_{n+1}(\zeta)\right)\right|}{\Vert L_1(\zeta)\Vert\cdots\Vert L_{n+1}(\zeta)\Vert}
=\lim_{k\to\infty}\frac{\left|\det\left(H_{1,k}(\zeta),\dots, H_{(n+1),k}(\zeta)\right)\right|}{\Vert H_{1,k}(\zeta)\Vert\cdots\Vert H_{(n+1),k}(\zeta)\Vert} \notag\\
&\quad=\lim_{k\to\infty}\frac{\left|\det\left(H_ {1\alpha_{k}}(z_{k} + \rho_{k} u_{k}\zeta),\dots, H_ {(n+1)\alpha_{k}}(z_{k} + \rho_{k} u_{k}\zeta)\right)\right|}{\Vert H_ {1\alpha_{k}}(z_{k} + \rho_{k} u_{k}\zeta)\Vert\cdots \Vert H_ {(n+1)\alpha_{k}}(z_{k} + \rho_{k} u_{k}\zeta)\Vert}\notag\\
&\quad\geq\lim_{k\to\infty}\left[ d_{FS}(H_ {1\alpha_{k}}(z_{k} + \rho_{k} u_{k}\zeta),\dots, H_ {(n+1)\alpha_k}(z_{k} + \rho_{k} u_{k}\zeta)\right]^n\notag\\
&\quad\geq \delta^n.
\end{align*}
 This implies that $\det\left(L_1(\zeta),\dots, L_{n+1}(\zeta)\right)$ is nowhere vanishing, and
\begin{align}\label{8.1}
\log\left|\det\left(L_1(\zeta),\dots, L_{n+1}(\zeta)\right)\right|\geq\sum_{i=1}^{n+1}\log\Vert L_i(\zeta)\Vert+n\log\delta.
\end{align}
Applying integration on both sides of (\ref{8.1}) and using Jensen's Lemma, we get
\begin{align*}
0=N(r,\left(\det\left(L_1,\dots, L_{n+1}\right)\right)_0)\geq \sum_{i=1}^{n+1}T_{L_i}(r)-O(1).
\end{align*}
for all $r>1.$ This contradicts to the assumption that $L_1$ is nonconstant.
\end{proof}
The  proof of Theorem \ref{Normal} is  based on the Zalcman lemma, Lemma \ref{tan}, and a notice given by Green in \cite{G}.   In \cite{T}, we also given some applications of Lemma \ref{tan} in the normal problem    concerning the condition of uniform boundedness  of  tangent mappings.\\

\noindent{\bf Proof of Theorem \ref{Normal}.}
Suppose that $\mathcal F$ is not normal, then by Lemma~\ref{Zalcman} there exist sequences $\{z_k\}\subset D$ with $z_k\to z_0\in D,$  $\{f_k\}\subset\mathcal F$, $\{\rho_k\}\subset\R$ with $\rho_k\to 0^+,$ and
Euclidean unit vectors ${u_k}\subset \C^m,$ such that $g_k(\zeta) := f_ k(z_k + \rho_k u_k\zeta),$ where $\zeta\in\C$
satisfies $z_k+ \rho_k u_k\zeta\in D,$ converges uniformly on compact subsets of $\C$ to a nonconstant
holomorphic mapping $g$ of $\C$ into $P^n(\C).$

 By Lemma \ref{tan},   $\{H_{1f_k}\}_{k=1}^\infty,\dots, \{H_{(2n+1)f_k}\}_{k=1}^\infty$ are normal families on $D.$
By replacing by subsequences, without loss of generality, we  assume that $\{H_{jf_k}\}_{k=1}^\infty$ ($1\leq j\leq 2n+1)$ converges uniformly on compact subsets of $D$ to a nonconstant
holomorphic mapping $h_j$ of $D$ into $P^n(\C).$

We take  reduced representations $\widehat {h_j}=(a_{j0},\dots, a_{jn})$ of $h_j$, $\widehat {f_k}=(f_{k0},\dots, f_{kn})$ of $f_k$ and $\widehat{H_{jf_k}}=(a_{jk0},\dots,a_{jkn})$ of $H_{jf_k}$ $(j=1,2,\dots,2n+1)$  in a neighbourhood $V_{z_0}$ of $z_0$ such that $\{a_{jki}\}_{k=1}^\infty$ converges uniformly  on  compact subsets of  $V_{z_0}$ to $a_{ji}$ ($i=0,\dots,n)$.

We consider hyperplanes $H_j: a_{j0}(z_0)\omega_0+\cdots+a_{jn}(z_0)\omega_n=0\;  (j=1,\dots,2n+1)$ in $P^n(\C).$

 Take a closed ball $B(z_0, R)=\{z: \Vert z-z_0\Vert \leq R\}\subset D.$
 By the assumption and by Lemma \ref{Cherry}, there is a positive constant $\delta_{B(z_0,R)}$ such that for all subsets $\{j_0,\dots, j_n\}\subset\{1,\dots,2n+1\}$ we have
 \begin{align*} \frac{\left|\det (a_{j_si}(z_0))_{0\leq s,i\leq n})\right|}{\Vert H_{j_0}\Vert\cdots \Vert H_{j_n}\Vert}&=\lim_{k\to\infty}
 \frac{\left|\det \left(H_{j_0f_k}(z_0),\dots,H_{j_nf_k}(z_0)\right)\right|}{\Vert H_{j_0f_k}(z_0)\Vert\cdots\Vert H_{j_nf_k}(z_0)\Vert}\\
 &\geq [d_{FS}(H_{j_0f_k}(z_0),\dots,H_{j_nf_k}(z_0))]^n\\
 &\geq \delta_{B(z_0, R)}^n>0.
 \end{align*}
Hence, $\det (a_{j_si}(z_0))_{0\leq s,i\leq n})\ne 0,$ for all subsets $\{j_0,\dots, j_n\}\subset\{1,\dots,2n+1\}$. Therefore, $H_1,\dots,H_{2n+1}$ are in general position.

For each $j\in\{1,\dots,2n+1\},$ by Hurwitz's theorem $g(\C)\subset H_j$ or $g(\C)\cap H_j=\varnothing;$ this is impossible, by the notice given by Green (\cite{G}, p. 112), there are no non-constant holomorphic maps of $\C$ into $(H_{i_1}\cap\cdots \cap H_{i_p})\setminus(H_{i_{p+1}}\cup\cdots\cup H_{i_{2n+1}})$, where $(i_1,\dots,i_{2n+1})$ is a permutation of $(1,\dots,2n+1)$.

We have completed the proof of Theorem \ref{Normal}.
\hfill$\square$
\section{The high-dimensional version of Lappan's theorem}
Let $ f = (f_0 :\cdots: f_n)$ be a holomorphic map from a domain in $\C$ to $P^n(\C)$ given by homogeneous coordinate functions $ f_j$ which are holomorphic without common zeros.  We have the following formula  for the Fubini-Study derivative $f^{\#}$ of $f$ (for details, see \cite{CE})
\begin{align}\label{2020}(f^{\#})^2:=\frac{\partial^2}{\partial z \partial \overline{z}}\log\sum_{i=0}^n|f_i|^2=\frac{\sum\limits_{0\leq s<t\leq n}\begin{vmatrix}f_s&f_t\\f_s'&f_t'\end{vmatrix}^2}{\Vert f\Vert^4}.
\end{align}
We shall prove the following high-dimensional version of Lappan's theorem.
\begin{theorem}\label{Lappan} Let $f$ be a holomorphic mapping of $\triangle$ into $P^n(\C),$ and let $H_1,H_2,\dots,H_q$ be hyperplanes in general position  in $P^n(\C).$ Assume that
\begin{align*}\sup\{(1-|z|^2)f^{\#}(z): z\in \cup_{j=1}^q f^{-1}(H_j)\}<\infty.
\end{align*}
If $q\geq n(2n+1)+2$  then $f$ is normal.
\end{theorem}
The counterpart result for holomorphic mappings of $\C$ into  $P^n(\C).$
\begin{theorem}\label{L} Let $f$ be a holomorphic mapping of $\C$ into $P^n(\C),$ and let $H_1,H_2,\dots,H_q$ be hyperplanes in general position  in $P^n(\C).$ Assume that
\begin{align*}\sup\{f^{\#}(z): z\in \cup_{j=1}^q f^{-1}(H_j)\}<\infty.
\end{align*}
If $q\geq n(2n+1)+2$, then $f^{\#}$ is upper bounded on $\C$.
\end{theorem}

\begin{lemma}[\cite{H}, Theorem 4] \label{Hahn} The holomorphic mapping $f:\triangle\to P^n(\C)$ is not normal if and only if there exist  sequences $\{z_k\}\subset\triangle,$ $\{r_k\}\subset\R,$ $r_k>0,$ with $\lim_{k\to\infty}\frac{r_k}{1-|z_k|}=0$ such that $g_k(\xi):=f(z_k+ r_k\xi)$ converges  uniformly on compact subsets of $\C$ to a non-constant holomorphic mapping $g$ of $\C$ into $P^n(\C).$
\end{lemma}

\begin{lemma}\label{SMT} Let $f$ be a linearly non-degenerate  holomorphic mapping of $\C$ into $P^n(\C).$  Let $H_1,\dots,H_q$ be $q$ hyperplanes  in $\kappa$-subgeneral position in $P^n(\C)$, where $\kappa\geq n$ and $q\geq 2\kappa-n+1$. Assume that    $f^{\#}=0 $ on $\cup_{j=1}^qf^{-1}(H_j)$.   Then $q\leq 2\kappa(n+1)-n+1.$
\end{lemma}
\begin{proof} Let $(f_0,\dots,f_n)$ be a reduced presentation of $f$.  For each $a\in\cup_{j=1}^qf^{-1}(H_j)$, we define
$$\Cal C_a:=\{(c_0,\dots,c_n)\in\C^{n+1}: c_0f_0(a)+\cdots+c_nf_n(a)=0\}.$$
Since $\Cal C_a$ is a vector subspace of dimension $n$ of $\C^{n+1}$ and since $\cup_{j=1}^qf^{-1}(H_j)$  is at most countable, it follows that there exists
\begin{equation*}
(c_0,\dots, c_n)\in \C^{n+1}\setminus(\cup_{a\in\cup_{j=1}^q f^{-1}(H_j)}\Cal C_a).
\end{equation*}

Let $L_0,\dots,L_n$ be $n+1$ hyperplanes in general position in $P^n(\C),$ where $L_0$ is defined by the equation: $c_0\omega_0+\cdots+c_n\omega_n=0.$\\
 By our choice for $(c_0,\dots, c_n)$
 \begin{align}
f^{-1}(L_0)\cap (\cup_{j=1}^qf^{-1}(H_j))=\varnothing.\label{t1}
\end{align}
Set $F:=(L_0(f):\cdots: L_n(f)):\C\to\P^n(\C).$ Then, $F$ is linearly nondegenerate and $T_F(r)=T_f(r)+O(1).$

Since $f^{\#}$ vanishes on $ \cup_{j=1}^qf^{-1}(H_j)$, we have $$(f_0:\cdots:f_n)=(f'_0:\cdots:f'_n)\;\text{ on }\; \cup_{j=1}^qf^{-1}(H_j).$$
Hence,
\begin{align}(L_0(f):\cdots: L_n(f))=\left((L_0(f))':\cdots: (L_n(f))'\right)\;\text{ on}\; \cup_{j=1}^qf^{-1}(H_j).\label{t2}
\end{align}
Since $F$ is  linearly nondegenerate, the Wronskian of $F$ is not identically equal to zero. Therefore,  there exists $t\in\{1,\dots,n\}$ such that
\[\det\begin{pmatrix}L_0(f)&L_{t}(f)\\
(L_0(f))'&(L_{t}(f))'\end{pmatrix}\not\equiv 0,\;\;\text{hence,}\left(\frac{L_{t}(f)}{L_0(f)}\right)'\not\equiv 0 .\]
By (\ref{t1}) and (\ref{t2}), we have
\begin{align}
\left(\frac{L_{t}(f)}{L_0(f)}\right)'=0\;\text{ on}\; \cup_{j=1}^qf^{-1}(H_j).\label{t3}
\end{align}
From the first main theorem and the lemma on logarithmic derivative of Nevanlinna theory for meromorphic functions, we get easily  that
$$T_{\left(\frac{L_{t}(f)}{L_0(f)}\right)'}(r)\leq 2T_{\left(\frac{L_{t}(f)}{L_0(f)}\right)}(r)+o\left(T_{\left(\frac{L_{t}(f)}{L_0(f)}\right)}(r)\right).$$
On the other hand, for each $a\in\C,$ since $H_1,\dots, H_q$ are in $\kappa$-subgeneral position in $P^n(\C)$, it follows that   there are at most $\kappa$  of them passing  through  $f(a).$  Hence, by (\ref{t1}) and (\ref{t3}), we have
\begin{align*}\sum_{j=1}^qN^{[1]}_f(r,H_j)&\leq \kappa N_{\left(\frac{L_{t}(f)}{L_0(f)}\right)'}(r)\notag\\
&\leq \kappa T_{\left(\frac{L_{t}(f)}{L_0(f)}\right)'}(r)+O(1)\notag\\
&\leq2\kappa T_{\frac{L_{t}(f)}{L_0(f)}}(r)+o\left(T_{\frac{L_{t}(f)}{L_0(f)}}(r)\right)\notag\\
&\leq2\kappa T_F(r)+o(T_F(r))\notag\\
&=2\kappa T_f(r)+o(T_f(r)).
\end{align*}
Then, by Nochka's second main theorem, we have
\begin{align*}
\Big\Vert 2\kappa T_f(r)+o(T_f(r))&\geq \sum_{j=1}^qN^{[1]}_f(r,H_j)\\
&\geq\frac{1}{n} \sum_{j=1}^qN^{[n]}_f(r,H_j)\\
&\geq\frac{q-2\kappa+n-1}{n} T_f(r)-o(T_f(r)).
\end{align*}
Hence, $q\leq 2\kappa(n+1)-n+1.$
\end{proof}
\noindent{ \bf Proof of Theorem \ref{Lappan}.}
Suppose that $f$ is not normal, then by Lemma \ref{Hahn}, there exist  sequences $\{z_k\}\subset\triangle,$ $\{r_k\},$ $r_k>0,$ with $\lim_{k\to\infty}\frac{r_k}{1-|z_k|}=0$ such that $g_k(\xi):=f(z_k+ r_k\xi)$ converges  uniformly on compact subsets of $\C$ to a non-constant holomorphic mapping $g$ of $\C$ into $P^n(\C).$

 Without loss of the generality, we may assume that $g(\C)\not\subset H_j$ for all $j\in\{1,\dots,q_0\}$ and $g(\C)\subset H_j$ for all $j\in\{q_0+1,\dots,q\},$ for some $q_0\leq q.$ Denote by $\mathbb P$ the smallest subspace of $P^n(\C)$ containing $g(\C)$. Then $p:=\dim \mathbb P\geq 1$, and $g$ is a linearly non-degenerate  entire curve in $\mathbb P.$ Since $H_1,\dots,H_q$ are in general position, we have $q-q_0+p\leq n,$ furthermore,
  $H_1':=H_1\cap \mathbb P, \dots, H_{q_0}':=H_{q_0}\cap\mathbb P$ are hyperplanes in $n-(q-q_0)$-subgeneral position in $\mathbb P.$

   Since $q\geq n(2n+1)+2>2n+1$,  we have  $q_0>q_0-(q-q_0)-(q-2n-1)-p=2[n-(q-q_0)]-p+1.$

 We now prove that $g^{\#}(\xi)=0$ for all $\xi\in\cup_{j=1}^{q_0}g^{-1}(H_j)=\cup_{j=1}^{q_0}g^{-1}(H_j').$
To do this, we consider an arbitrary point  $\xi_0\in\cup_{j=1}^{q_0}g^{-1}(H_j').$   Take an index $j_0\in\{1,\dots,q_0\}$ such that  $\xi_0\in g^{-1}(H'_{j_0})=0.$ By Hurwitz's Theorem there are values $\{\xi_k\}$  (for all $k$ sufficiently large), $\xi_k\to\xi_0$ such that $\xi_k\in g_k^{-1}(H_{j_0}),$ and hence, $z_k+r_k\xi_k\in f^{-1}(H_{j_0}).$
Therefore, by the assumption, there is a positive constant $M$ such that
$$(1-|z_k+r_k\xi_k|^2)f^{\#}(z_k+r_k\xi_k)<M$$
for all $k$ sufficiently large. \\
  We have
\begin{align*}g^{\#}(\xi_0)&=\lim_{k\to\infty}g_k^{\#}(\xi_k)=\lim r_k f^{\#}(z_k+r_k\xi_k)\notag\\
&=\lim_{k\to\infty}\frac{r_k(\frac{1}{1-|z_k|}-|\frac{z_k}{1-|z_k|}+\frac{r_k}{1-|z_k|}\xi_k|)^{-1}}{(1-|z_k|)(1+|z_k+r_k\xi_k|)} (1-|z_k+r_k\xi_k|^2)f^{\#}(z_k+r_k\xi_k)\notag\\
&=0,
\end{align*}
(note that $\lim_{k\to\infty}\frac{r_k}{1-|z_k|}=0$ and $\frac{1}{1-|z_k|}-|\frac{z_k}{1-|z_k|}+\frac{r_k}{1-|z_k|}\xi_k|\geq(\frac{1}{1-|z_k|}-\frac{|z_k|}{1-|z_k|})-\frac{r_k}{1-|z_k|}|\xi_k|>\frac{1}{2}$, for all $k$ sufficiently large).\\
Hence,
\begin{align*}
g^{\#}=0\; \text{ in}\; \cup_{j=1}^{q_0}g^{-1}(H_j').
\end{align*}
Applying  Lemma \ref{SMT}, we have
$$q_0\leq 2(n-(q-q_0))(p+1)-p+1.$$
Then
$$q_0+2(p+1)(q-q_0)\leq 2n(p+1)-p+1.$$
Therefore,
\begin{align*}q&=q_0+(q-q_0)\\
&\leq q_0+2(p+1)(q-q_0)\\
&\leq 2n(p+1)-p+1\\
&\leq n(2n+1)+1.
\end{align*}
This contradicts to the assumption that $q\geq n(2n+1)+2.$

We have completed the proof of Theorem \ref{Lappan}.\hfill$\square$\\

\noindent Tran Van Tan\\
Department of Mathematics\\
  Hanoi National University of Education\\
 136-Xuan Thuy street, Cau Giay, Hanoi, Vietnam\\
e-mail: tantv@hnue.edu.vn
\end{document}